\documentclass[12pt,oneside]{article}
\usepackage[inner=32mm,outer=31mm,tmargin=30mm,bmargin=40mm]{geometry}
\usepackage[utf8]{inputenc}
\usepackage[T1]{fontenc}
\usepackage{sectsty}
    \sectionfont{\large}
    \subsectionfont{\normalsize}
\usepackage{amssymb}
\usepackage{geometry}
\usepackage{amsmath}
\usepackage{amsthm}
\usepackage{graphicx}
\usepackage{sidecap}
\usepackage{color}
\usepackage[normalem]{ulem}
\usepackage[english]{babel}
\usepackage{lipsum}
\usepackage[onehalfspacing]{setspace}
\usepackage{verbatim}
\usepackage[hidelinks]{hyperref}
\usepackage[inline]{enumitem}
\makeatletter
\newcommand{\inlineitem}[1][]{%
\ifnum\enit@type=\tw@
    {\descriptionlabel{#1}}
  \hspace{\labelsep}%
\else
  \ifnum\enit@type=\z@
       \refstepcounter{\@listctr}\fi
    \quad\@itemlabel\hspace{\labelsep}%
\fi}
\makeatother
\usepackage{algorithmic}
\usepackage{algorithm}

\newcommand{\IR}{\ensuremath{\mathbb{R}}}
\newcommand{\IN}{\ensuremath{\mathbb{N}}}

\newcommand{\norm}[1]{\left\Vert#1\right\Vert}
\newcommand{\set}[1]{\left\{#1\right\}}
\newcommand{\abs}[1]{\left|#1\right|}
\newcommand{\brackets}[1]{\left(#1\right)}
\renewcommand{\d}{{\rm d}} 
\renewcommand{\rho}{\varrho}

\definecolor{darkred}{RGB}{139,0,0}
\definecolor{darkgreen}{RGB}{0,100,0}
\definecolor{darkmagenta}{RGB}{139,0,139}
\definecolor{darkpurple}{RGB}{110,0,180}
\definecolor{darkblue}{RGB}{40,0,200}
\definecolor{darkorange}{RGB}{255,140,0}

\DeclareMathOperator{\cost}{cost}
\DeclareMathOperator{\er}{e}

\newtheorem{thm}{Theorem}

\theoremstyle{plain}
\newtheorem{lem}{Lemma}

\theoremstyle{definition}

\newtheorem{rem}{Remark}
\newtheorem{alg}{Algorithm}

\title{
Recovery algorithms for high-dimensional rank one tensors
} 

\author{David Krieg\\
Mathematisches Institut, Universit\"at Jena\\
Ernst-Abbe-Platz 2, 07743 Jena, Germany\\
email: 
david.krieg@uni-jena.de\\[2ex]
Daniel Rudolf\\
Institut f\"ur Mathematische Stochastik, Universit\"at G\"ottingen\\
Goldschmidtstr. 3-5, 37077 G\"ottingen, Germany\\
email: 
daniel.rudolf@uni-goettingen.de}

\date{\today}

\begin{document}

\maketitle

\begin{abstract}
\noindent
We present 
deterministic algorithms for the uniform recovery
of $d$-variate 
rank one tensors from function values.
These tensors are given as product of $d$ univariate functions
whose $r$th weak derivative is bounded by $M$. 
The recovery problem is known to suffer from the curse
of dimensionality for $M\geq  2^r r!$.
For smaller $M$, a randomized algorithm is known
which breaks the curse. 
We 
construct
a deterministic algorithm which is even less costly.
In fact, 
we completely characterize
the tractability of this problem 
by distinguishing
three different ranges of the parameter $M$.
\end{abstract}

\noindent
{\bf Keywords: } High dimensional approximation, rank one tensors, 
worst case error, tractability, curse of dimensionality, dispersion.

\noindent
{\bf Classification:} Primary: 65Y20; Secondary: 41A15, 41A25, 41A63, 65D99. 


\section{Introduction}  \label{sec1}

Suppose we know that a $d$-variate function $f$
is the product of $d$ univariate functions
with a certain smoothness.
How many function values do we need to capture $f$
up to some error $\varepsilon\in (0,1)$ in the uniform norm?
This question has been
posed and investigated 
in the work of Bachmayr, Dahmen, DeVore and Grasedyck \cite{BDDG}.
The hope is that the structural knowledge about $f$
allows for 
efficient deterministic approximation schemes
in high-dimensional settings. 
More precisely, 
it is assumed that $f$ is contained in 
the class of rank one tensors given by
\begin{equation*} 
F_{r,M}^d = \big\{ \bigotimes_{i=1}^d  f_i \mid
f_i:[0,1]\to [-1,1], \
\Vert f_i^{(r)} \Vert_\infty \le M \big\}
\end{equation*} 
for smoothness parameters $r\in \mathbb{N}$ and $M>0$.
Here, $f_i^{(r)}$ denotes the $r$th weak derivative of $f_i$.
In particular, 
it is assumed that $f_i$ is contained in
the class $W_\infty^r([0,1])$ of univariate functions
which have $r$ weak derivatives in $L_\infty([0,1])$.

It is proven in \cite{NoRu16} that for $M\geq2^r r!$ this problem
suffers from the curse of dimensionality:
To ensure an error smaller than $\varepsilon$,
any deterministic algorithm must use exponentially many function values
with respect to the dimension.
Even for randomized methods, the curse is present.
For $M<2^r r!$ however,
a randomized algorithm is constructed
which does not require exponentially many function values.
We are driven by the question whether 
the same is possible with a deterministic algorithm.
We give an affirmative answer to this question.
In fact, we explicitly construct and analyze deterministic algorithms
for different ranges of the smoothness parameters. 
We use the following terminology.

The worst case error of an algorithm $A$ on the class $F_{r,M}^d$ is given by
\[
 \er(A):=\sup_{f\in F_{r,M}^d} \norm{f-A(f)}_\infty.
\]
The number of function values used by $A$ for the input $f$ 
is denoted by $\cost(A,f)$. The worst case cost of $A$ is given by
\[
 \cost(A) := \sup_{f\in F_{r,M}^d} \cost(A,f).
\]

A deterministic algorithm is already constructed in \cite{BDDG}.
It achieves the worst case error $\varepsilon$
while using at most
\begin{equation*}
 C_{r,d}\, M^{d/r} \varepsilon^{-1/r}
\end{equation*}
function values of $f$, see \cite[Theorem~5.1]{BDDG}.
This number behaves optimally as a function of $\varepsilon$.
However, the constant $C_{r,d}$ and hence
the number of function values grows super-exponentially with $d$
for any $M>0$ and $r\in\IN$.
For the algorithm, the following observation of 
Bachmayr, Dahmen, DeVore and Grasedyck is crucial.
If we know some $z^*\in [0,1]^d$
with $f(z^*)\not =0$,  
we can
construct a
method
$I_{m}(z^*,\cdot)$
that uses $m$ function values and satisfies
\begin{equation} \label{eq: known_z_star}
 \norm{I_{m}(z^*,f)-f}_\infty \leq \varepsilon,
\end{equation}
if we choose
\begin{equation} \label{eq: n_chosen_for_err_eps}
   m= \left\lfloor C_{r,M}\, d^{1+1/r} \varepsilon^{-1/r} \right\rfloor.
\end{equation}
Here, $C_{r,M}$ is a positive constant which only 
depends on $r$ and $M$. For example, 
one can choose $C_{r,M}=4\max\{1,C_1(r)M\}^{1/r}$
with $C_1(r)$ as in \cite[Section 2]{BDDG}.
Roughly speaking, the knowledge of a non-zero of $f$
allows us to reduce the problem
to $d$ univariate approximation problems
which can, for example, be treated by the use of polynomial interpolation.
With this observation at hand, the authors of \cite{BDDG}
use an approximation scheme of the following type:
\begin{alg}
\label{generic algorithm}
   Given $m\in \mathbb{N}$, a finite point set $P\subset[0,1]^d$ 
   and a function $f\in F_{r,M}^d$,
   obtain $A_{P,m}(f)$ as follows:
\begin{enumerate}
 \item 
 For any $x\in P$ check whether $f(x)\not = 0$.
 \item If we found some $z^*\in P$ with $f(z^*)\not =0$ then
       call $I_{m}(z^*,f)$ from \eqref{eq: known_z_star}. 
       If $f_{\mid P}=0$,
       then return the zero function.
\end{enumerate}
\end{alg}
The idea behind this algorithm is to choose $P$ such that
whenever $f_{\mid P} = 0$, then $\Vert f\Vert_\infty$
must be small and the zero function is
a good approximation of $f$.
The authors of \cite{BDDG} use a point set $P$, which contains
a finite Halton sequence $H$.
They obtain that $f_{\mid P} = 0$ implies
\[
 \norm{f}_\infty \leq 
 (2M)^d \left(2^d  \pi_d\right)^r \abs{H}^{-r},
\]
where $\pi_d$ is the product of the first $d$ primes.
To ensure an error bound smaller than $\varepsilon$ one needs 
\begin{equation}  \label{eq: card_P_bei_BDDG}
 \abs{P} \geq \abs{H} \geq 
 (2^d M^d \varepsilon^{-1})^{1/r} \left(2^d \pi_d \right)
\end{equation}
function evaluations of $f$.
However, this number increases super-exponentially with the dimension
for all parameters $M$ and $r$.

We also use Algorithm~\ref{generic algorithm}
but for smaller point sets $P$.
To give a better intuition
of the role of the point set, we 
introduce the notion of detectors.
We call a finite point set $P$ in $[0,1]^d$ 
an $\varepsilon$-detector for the class $F_{r,M}^d$
if it contains (detects) a non-zero of every function $f\in F_{r,M}^d$
with uniform norm greater than $\varepsilon$.
If $P$ is an $\varepsilon$-detector for $F_{r,M}^d$
and $m$ is chosen as in (\ref{eq: n_chosen_for_err_eps}),
it is easy to see 
that Algorithm~\ref{generic algorithm} satisfies
\begin{equation*}
 \er(A_{P,m}) \leq \varepsilon \quad\text{and}\quad
 \cost(A_{P,m}) \leq \abs{P} + m,
\end{equation*}
see Lemma~\ref{detector lemma}.
We thus need to construct small $\varepsilon$-detectors $P$ for $F_{r,M}^d$.

In the range $M\geq 2^r r!$
we know that the problem suffers from the curse of dimensionality
such that we cannot expect 
to find an $\varepsilon$-detector with small cardinality. 
In that case, we provide a detector 
for which the cardinality of the point set depends 
exponentially on $d$.
In the range $M<2^r r!$
we give a detector whose cardinality only grows polynomially with $d$.
The order of growth is proportional to $\log_2(\varepsilon^{-1})$.
For $M\leq r!$
the point set can be chosen even smaller.
There is a detector whose cardinality grows quadratically 
with $d$, at worst,
regardless of the value of $\varepsilon$.
Altogether, we obtain the following:
\begin{thm}
\label{thm: thm_upp_bnd}
 For any $r\in\mathbb{N}$ and $M>0$, there are positive constants
 $c_i$, $i\leq 4$, such that the following holds.
 For any $d\in \mathbb{N}$ and $\varepsilon\in (0,1)$,  
 there is a finite point set $P\subset[0,1]^d$
 and a natural number $m$ such that $\er(A_{P,m})\leq \varepsilon$ and
 \begin{equation*}
  \cost(A_{P,m}) \leq
  \left\{\begin{array}{lr}
        c_1^d\, \varepsilon^{-1/r}
        \quad
        & \text{if } M\in(0,\infty),\\ 
        c_2 \exp\left(c_3(1+\ln(\varepsilon^{-1}))(1+\ln d)\right)\quad
        & \text{if } M\in(0,2^r r!),\\
        c_4\, d^{2} \varepsilon^{-1/r} \ln(\varepsilon^{-1/r})
        & \text{if } M\in(0,r!].
        \end{array}\right.
 \end{equation*}
\end{thm}
We always choose $m$ as in \eqref{eq: n_chosen_for_err_eps}.
The point sets $P$ and the constants $c_i$ can be found in Section~\ref{algorithms section}.
In each of these ranges
we also give a lower bound on the complexity of the problem,
which is the reason for
us to call the resulting algorithms optimal.
In particular, we obtain the following tractability results.
We use standard notions of tractability, 
see Section~\ref{lower bounds section} for their definition.

\begin{thm}
\label{tractability results}
 The problem of the uniform approximation
 of functions in $F_{r,M}^d$ with deterministic algorithms
 based on function values suffers from the curse of dimensionality,
 if and only if $M\in[2^r r!,\infty)$.
 If $M\in(r!,2^r r!)$, it is quasi-polynomially tractable
 but not polynomially tractable.
 If $M\in(0,r!]$, it is polynomially tractable
 but not strongly polynomially tractable.
\end{thm}
Note that the first two statements are also true for
randomized algorithms.

Before we proceed to the proofs,
let us introduce some further notation that is used in the paper.
In the following, the term \emph{box} 
refers to a product of $d$ nonempty subintervals $I_j$ of $[0,1]$, in formulas it takes the form $\prod_{j=1}^d I_j$.
The \emph{dispersion} of a finite subset $P$ of $[0,1]^d$ is the minimal
number $\eta>0$ such that 
$P$ has non-empty intersection with
every box of volume greater than $\eta$.
For any $k\in\IN$, the set of natural numbers up to $k$ 
is denoted by $[k]$.
If $x_i$ is a real number for each $i$ in some finite set $J$,
we set $x_J=(x_i)_{i\in J}$.
We write $x_J=\mathbf{1}$ if $x_j=1$ for all $j\in J$.
If $I_i$ is an interval for each $i\in J$, then $I_J$ 
denotes the Cartesian product of these intervals.
If we are given functions $f_i:I_i\to \IR$ for each $i\in J$,
their tensor product is denoted by $f_J:I_J\to \IR$.
Throughout the paper, $r$ and $d$ are natural numbers,
$\varepsilon$ is an element of $(0,1)$ and $M$ is positive.
The natural logarithm of a positive number $a$ is denoted by $\ln a$,
its logarithm in base two by $\log_2 a$.

\section{Algorithms}
\label{algorithms section}

We start with the observation 
that the construction of an $\varepsilon$-detector
is sufficient to achieve the worst case error $\varepsilon$
with the algorithm $A_{P,m}$.
Recall that a point set $P$ in $[0,1]^d$ is called
an \mbox{$\varepsilon$-detector} for $F_{r,M}^d$,
if it contains a non-zero of any function
$f\in F_{r,M}^d$ with $\norm{f}_\infty > \varepsilon$.
Note that any such function is of the following form:
\begin{equation}
\label{target function}
\begin{split}
 f= \bigotimes\limits_{i=1}^d f_i,
 \quad &\text{where} \quad f_i:[0,1]\to[-1,1]
 \quad \text{with} \quad \Vert f_i^{(r)}\Vert_\infty \leq M \\
 &\text{and} \quad \norm{f}_\infty = \prod\limits_{i=1}^d \norm{f_i}_\infty > \varepsilon.
\end{split}
\end{equation}

\begin{lem}
\label{detector lemma}
 Let $r\in\IN$, $d\in\IN$ and $M>0$.
 If $P$ is an $\varepsilon$-detector for $F_{r,M}^d$
 and $m$ is chosen as in (\ref{eq: n_chosen_for_err_eps}), 
 then
 Algorithm~\ref{generic algorithm} satisfies
 \begin{equation*}
  \er(A_{P,m}) \leq \varepsilon \quad\text{and}\quad 
  \cost(A_{P,m}) \leq \abs{P} + m.
 \end{equation*}
\end{lem}
\begin{proof}
Let $f\in F_{r,M}^d$.
If $P$ contains a non-zero of $f$,
Algorithm~\ref{generic algorithm}
returns an \mbox{$\varepsilon$-approximation} of $f$
due to relation (\ref{eq: known_z_star}).
If not, the output is zero.
But since $P$ is a detector, we necessarily have
$\Vert f\Vert_\infty\leq \varepsilon$ and zero
is an $\varepsilon$-approximation of $f$, as well.
The second statement is obvious.
\end{proof}

Furthermore, we will use the following formula
for polynomial interpolation.

\begin{lem}
 \label{interpolation lemma}
 Let $a<b$, $r\in\IN$ and $g\in W_\infty^r([a,b])$.
 Let $x_1,\hdots,x_r\in[a,b]$ be distinct
 and $p$ be the unique polynomial with degree less than $r$
 such that $p(x_i)=g(x_i)$ for all $i\in[r]$.
 For every $x\in[a,b]$, 
 there exist $\xi_1,\xi_2\in[a,b]$ such that
 $$
 g(x)-p(x) =
 \frac{1}{r!} \cdot
 \frac{g^{(r-1)}\brackets{\xi_2}-g^{(r-1)}\brackets{\xi_1}}{\xi_2-\xi_1} \cdot
 \prod\limits_{i=1}^r \brackets{x-x_i} .
 $$
\end{lem}

Lemma~\ref{interpolation lemma} is well known for $g\in\mathcal{C}^r([a,b])$.
In this case, the second fraction can be 
replaced by $g^{(r)}(\xi)$ for some $\xi\in [a,b]$.
We refer to \cite[Theorem~2, Section~6.1]{KiCh91}.
Under the more general assumption that $g\in W_\infty^r([a,b])$,
we have to modify the proof of the mentioned theorem.

\begin{proof}
 If $x$ coincides with one of the nodes, the statement is trivial.
 Hence, let $x$ be distinct from all the nodes.
 We consider
 $$
 w: [a,b]\to \IR, \quad w(y)=\prod\limits_{i=1}^r \brackets{y-x_i}
 $$
 and set
 $$
 \lambda = \frac{g(x)-p(x)}{w(x)}.
 $$
 The function $\phi=g-p-\lambda w$
 vanishes at the points $x_1,\hdots,x_r$ and $x$.
 Since $g$ and $\phi$ are elements of $W_\infty^r([a,b])$,
 their $(r-1)^{\text{st}}$ derivatives are absolutely continuous.
 If we apply Rolle's Theorem $(r-1)$ times,
 we obtain that $\phi^{(r-1)}$ has at least 2 distinct zeros $\xi_1$
 and $\xi_2$ in $[a,b]$ and hence
 $$
 0=\int_{\xi_1}^{\xi_2} \phi^{(r)}(y)~\d y
 =\int_{\xi_1}^{\xi_2} g^{(r)}(y)-\lambda r! ~\d y
 = g^{(r-1)}\brackets{\xi_2}-g^{(r-1)}\brackets{\xi_1}
 - \lambda r! \brackets{\xi_2-\xi_1}.
 $$
 This is the stated identity in disguise.
\end{proof}
If $g\in W_\infty^r([0,1])$ has $r$ distinct zeros $x_1,\dots,x_r\in [0,1]$, 
and $x$ is a maximum point of $\abs{g}$, we get
\begin{equation} \label{eq: polyn_interpol}
 \norm{g}_\infty \leq \frac{\norm{g^{(r)}}_\infty}{r!} \prod_{i=1}^r \abs{x-x_i}.
\end{equation}
This follows from Lemma~\ref{interpolation lemma} since the unique polynomial
$p$ with degree less than $r$ and $p(x_i)=g(x_i)$ for $i\in [r]$ is the zero polynomial
and
\[
\abs{g^{(r-1)}\brackets{\xi_2}-g^{(r-1)}\brackets{\xi_1}} =
\abs{\int_{\xi_1}^{\xi_2} g^{(r)}(y)~\d y}
\leq \norm{g^{(r)}}_\infty\cdot \abs{\xi_2-\xi_1}.
\]

The rest of this section is devoted to
the construction of small $\varepsilon$-detectors for $F_{r,M}^d$.
Thanks to Lemma~\ref{detector lemma},
this is sufficient to prove Theorem~\ref{thm: thm_upp_bnd}.
We will use three different strategies for
three different ranges of the parameter $M$.

\subsection{Detectors for large derivatives}
\label{Large Derivatives}

In this section, the smoothness parameter $M$ can be arbitrarily large.
It is shown in \cite{NoRu16} that
the cost of any algorithm with worst case error smaller than $1$
is at least $2^d$ if $M\geq 2^r r!$.
In particular, the cardinality of 
any \mbox{$\varepsilon$-detector} must grow exponentially with the dimension.
Yet, it does not get any worse:
We construct an \mbox{$\varepsilon$-detector} whose cardinality
``only'' grows exponentially with the dimension
but not super-exponentially.
We use the following lemma.

\begin{lem}
\label{empty interval lemma}
 For each $g\in W_\infty^r([0,1])$ with $\Vert g^{(r)}\Vert_\infty \leq M$
 there is a subinterval of $[0,1]$ with length 
 \begin{equation*}
  L(g)=\min\set{\frac{1}{r},\brackets{\frac{\norm{g}_\infty}{M}}^{1/r}}
 \end{equation*}
 which does not contain any zero of $g$.
\end{lem}

\begin{proof}
 The function $\abs{g}$ attains its maximum, say for $x\in [0,1]$.
 We choose an interval $I\subset [0,1]$ of length $rL(g)$ that contains $x$.
 There are $r$ open and disjoint 
 subintervals of $I$ with length $L(g)$.
 We label these intervals $I_1,\dots,I_r$ such that the distance 
 of $x$ and $I_i$ is increasing with $i$.
 Assume that every interval $I_i$ contains a zero $x_i$ of $g$.
 Then we have $\abs{x-x_i}< iL(g)$ for all $i\in [r]$
 and \eqref{eq: polyn_interpol} leads to
 \[
  \norm{g}_\infty
  \leq \frac{M}{r!} \prod_{i=1}^r \abs{x-x_i} 
  < M L(g)^r
  \leq \norm{g}_\infty.
 \]
 This is a contradiction and the assertion is proven.
\end{proof}

If, in addition, the uniform norm of $g$ is bounded by $1$, we have
\begin{equation*}
 L(g)\geq \varrho^{-1}\norm{g}_\infty^{1/r}  \quad\text{for}\quad
 \varrho=\max\set{r,M^{1/r}}.
\end{equation*}
Hence, for every $f$ satisfying \eqref{target function}
there is a box $B$ in $[0,1]^{d}$ with volume
\begin{align*}
 &\prod_{i\in [d]} L\brackets{f_i}
 \geq \varrho^{-d} \prod_{i\in [d]} \norm{f_i}_\infty^{1/r}
 = \varrho^{-d} \norm{f}_\infty^{1/r}
 > \varrho^{-d} \varepsilon^{1/r}
\end{align*}
such that $f$ does not vanish anywhere on $B$.
Hence, any point set $P$ in $[0,1]^{d}$ with dispersion
$\varrho^{-d} \varepsilon^{1/r}$ or less 
is an $\varepsilon$-detector for $F_{r,M}^d$.
We know from the estimate of Larcher, see \cite{ahr}, 
that we can choose $P$ as a $(t,s,d)$-net with cardinality
\begin{equation*}
 \abs{P} = \left\lceil 2^{7d+1} \varrho^d \varepsilon^{-1/r}\right\rceil.
\end{equation*}
By Lemma~\ref{detector lemma},
the resulting algorithm achieves the worst case error $\varepsilon$ with 
the cost
\begin{equation*}
 \cost\brackets{A_{P,m}}\leq
 \left\lceil 2^{7d+1} \varrho^d \varepsilon^{-1/r} \right\rceil
 + C_{r,M}\, d^{1+1/r} \varepsilon^{-1/r}.
\end{equation*}
This proves the first statement of Theorem~\ref{thm: thm_upp_bnd} 
with $c_1=2^8\rho+C_{r,M}$.
Note that the cost of this algorithm
has the minimal order of growth with respect to $\varepsilon$.
It grows like $\varepsilon^{-1/r}$ if $d$ is fixed
and $\varepsilon$ tends to zero.

\subsection{Detectors for moderately large derivatives}
\label{Moderate Derivatives}

In this section, we assume that $M < 2^r r!$.
In this case, we construct detectors $P$
with a cardinality that only grows polynomially with $d$
for any fixed $\varepsilon$.
The construction of $P$
is based on the observation that for any
function $f$ from $\eqref{target function}$
only some of the factors $f_i$
can have more than $(r-1)$ zeros close to $1/2$.
This is 
an essential difference
to the case $M\in [2^r r!,\infty)$,
where all factors $f_i$ may have infinitely many zeros
in any neighborhood of $1/2$.
We are going to specify this statement in Lemma~\ref{harmless functions},
but first we need the following observation.
For $\delta\in (0,1/2]$, we consider the interval
$I_\delta := [1/2-\delta,1/2+\delta]$.

\begin{lem}
\label{zeros at 1/2}
 Let $g\in W_\infty^r([0,1])$ with $\Vert g^{(r)}\Vert_\infty \leq M$.
  Assume that $g$ has $r$ distinct zeros in
  $I_\delta$. Then
  \[
  \norm{g}_\infty\leq C_\delta := \frac{M(1+2\delta)^r}{2^r r!}.
  \]
\end{lem}

\begin{proof}
 Let $x_1,\dots,x_r$ be those zeros.
 The function $\abs{g}$ attains its maximum, say for $x\in [0,1]$.
 By \eqref{eq: polyn_interpol} we have
 \[
  \norm{g}_\infty
  \leq \frac{\norm{g^{(r)}}_\infty}{r!} \prod_{i=1}^r \abs{x-x_i}.
 \]
 This yields the desired inequality since $\abs{x-x_i}\leq 1/2+\delta$ for each $i\in [r]$.
\end{proof}

Since $M < 2^r r!$, we can choose $\delta\in(0,1/2]$ such that $C_\delta <1$.
We define the pseudo-dimension
$d_0$ as the largest number in $[d]\cup\{0\}$ that satisfies
$C_\delta^{d_0} >\varepsilon$, that is
\begin{equation*}
 d_0 := \min\set{\left\lceil \frac{\ln \varepsilon}{\ln C_\delta}\right\rceil -1, d}.
\end{equation*}
Obviously, the pseudo-dimension is bounded above independently of $d$.
We can now specify the statement from the beginning of this section.

\begin{lem}
\label{harmless functions}
 Let $f$ be given as in (\ref{target function}).
 Then there are at most $d_0$ coordinates $i\in [d]$
 such that $f_i$ has more than $(r-1)$ zeros in $I_\delta$.
\end{lem}

\begin{proof}
 Let $k$ be the number of coordinates $i\in[d]$
 for which $f_i$ has more than $(r-1)$ zeros in $I_\delta$.
 Lemma~\ref{zeros at 1/2} yields that $\varepsilon < \Vert f\Vert_\infty \leq C_\delta^k$.
 The maximality of $d_0$ yields that $k\leq d_0$.
\end{proof}

This means that there is a subset $J^*$ of $[d]$ with cardinality $d_0$
such that $f_i$ has at most $(r-1)$ zeros in $I_\delta$ for all $i\in [d]\setminus J^*$.
We can find a non-zero of $f$, if we solve the following tasks:
\begin{enumerate}
 \item Find $J^*$.
 \inlineitem Find a non-zero of $f_{J^*}$.
 \inlineitem Find a non-zero of $f_{[d]\setminus J^*}$.
\end{enumerate}
Let us go through these tasks one by one.
We will deal with the first task by simply going through all possible sets
$J\subset [d]$ of cardinality $d_0$.
The number of such sets only depends polynomially on $d$.
We can cope with the second task,
since this problem is only $d_0$-dimensional.
By Lemma~\ref{empty interval lemma}, 
there is a box $B$ in $[0,1]^{d_0}$ with volume
\begin{align*}
 &\prod_{i\in J^*} L\brackets{f_i}
 \geq \prod_{i\in J^*} \varrho^{-1} \norm{f_i}_\infty^{1/r}
 \geq \varrho^{-d_0} \prod_{i\in [d]} \norm{f_i}_\infty^{1/r}
 = \varrho^{-d_0} \norm{f}_\infty^{1/r}
 > \varrho^{-d_0} \varepsilon^{1/r}
\end{align*}
such that $f_{J^*}$ does not vanish on $B$.
Hence, any point set $P_1$ in $[0,1]^{d_0}$ with dispersion
$\varrho^{-d_0} \varepsilon^{1/r}$ or less contains a non-zero of $f_{J^*}$.
Again by the result of Larcher, see \cite{ahr}, 
we know that we can choose $P_1$ as a $(t,s,d)$-net
of cardinality $2^{7d_0+1} \varrho^{d_0} \varepsilon^{-1/r}$.
The third task is also feasible,
since $f_i$ has at most $(r-1)$ zeros in $I_\delta$
for all $i\in [d]\setminus J^*$.
We use the following observation.

\begin{lem}
\label{diagonal lemma}
 Let $J$ be an $\ell$-element subset of $[d]$
 and, for every $i\in J$, let $f_i$ be a function with at most
 $k$ zeros on some interval $I_i$.
 Then every $(\ell k+1)$-element set in $I_J$
 whose elements are pairwise distinct in every coordinate
 contains a non-zero of $f_J$.
\end{lem}

\begin{proof}
 Let $P$ be an $(\ell k+1)$-element set in $I_J$
 whose elements are pairwise distinct in every coordinate 
 and suppose that $f_J$ vanishes everywhere on $P$.
 For each $i\in J$ let
 $P_i=\set{x_J\in P \mid f_i(x_i)=0}$.
 Since $f_J(x_J)=0$ implies that there is some $i\in J$ with $f_i(x_i)=0$, we have
 $P=\bigcup_{i\in J} P_i$.
 This can only be true, if one of the sets $P_i$ has more than $k$ elements.
 But since $x_i$ is different for every $x_J\in P_i$,
 this means that the corresponding function $f_i$ has more than $k$ zeros, a contradiction.
\end{proof}

We 
apply this lemma
for the functions $f_i$ in~\eqref{target function},
for the index set $J=[d]\setminus J^*$ and
for $k=r-1$. We obtain that the diagonal set
\begin{equation*}
 P_2=\set{\brackets{\frac{1}{2} -\delta + \frac{2\delta j}{(r-1)(d-d_0)}}\cdot \boldsymbol 1
 \mid j\in\IN_0 \text{ with } j\leq (r-1)(d-d_0)}
\end{equation*}
in $[0,1]^{d-d_0}$ contains a non-zero of $f_{[d]\setminus J^*}$.
All together, we obtain the detector
\begin{equation*}
P=\bigcup_{J\subset [d]:\, \abs{J}=d_0}
 \set{x\in [0,1]^d \mid x_J\in P_1,\, x_{[d]\setminus J}\in P_2 }.
\end{equation*}
In fact, we have seen above that for any $f$ satisfying~\eqref{target function}
there must be some $J^*\subset [d]$ with $\abs{J^*}=d_0$,
a non-zero $y\in P_1$ of $f_{J^*}$
and a non-zero $z\in P_2$ of $f_{[d]\setminus J^*}$.
The point $x\in [0,1]^d$ with $x_{J^*}=y$ and $x_{[d]\setminus J^*}=z$
is contained in the set $P$ and a non-zero of $f$.
The cardinality of the detector is given by
\begin{equation*}
\abs{P}
 = \binom{d}{d_0} \abs{P_1} \abs{P_2}
 = \binom{d}{d_0} \left[(r-1)(d-d_0)+1\right] 2^{7d_0+1} \varrho^{d_0} \varepsilon^{-1/r}.
\end{equation*}
This number grows like $d^{d_0+1}$
if $\varepsilon$ is fixed and $d$ tends to infinity.
Together with Lemma~\ref{detector lemma},
this proves the second statement of Theorem~\ref{thm: thm_upp_bnd} with
\[
 c_2=2r+C_{r,M},
 \quad\text{and}\quad
 c_3=\ln(2^7\rho)\left(1+1/\ln(C_\delta^{-1})\right).
\]
Note that 
$d_0$ equals $d$ 
if $\varepsilon$ is small enough.
Hence, the cardinality of $P$ and the cost 
of the algorithm grows like $\varepsilon^{-1/r}$
if $d$ is fixed and $\varepsilon$ tends to zero,
which is optimal.

\subsection{Detectors for small derivatives}
\label{Small Derivatives}

In this section, we assume that $M\leq r!$.
In this case, each function $f$ satisfying~\eqref{target function} 
does not vanish almost everywhere on a box whose size is independent of $d$.
This is due to the following fact.

\begin{lem}
\label{almost empty interval lemma}
 For each $g\in W_\infty^r([0,1])$ with $\Vert g \Vert_\infty \leq 1$
 and $\Vert g^{(r)}\Vert_\infty \leq r!$ 
 there is an interval in $[0,1]$ with length 
 $\norm{g}_\infty^{1/r}$
 which contains at most $(r-1)$ zeros of $g$.
\end{lem}

\begin{proof}
 The function $\abs{g}$ attains its maximum, say for $x\in [0,1]$.
 We choose an open interval $I\subset [0,1]$ 
 of length $\norm{g}_\infty^{1/r}$ whose closure contains $x$.
 Assume that $I$ contains $r$ distinct zeros $x_1,\hdots,x_r$ of $g$.
 Then $\abs{x-x_i}< \norm{g}_\infty^{1/r}$ for all $i\in [r]$
 and \eqref{eq: polyn_interpol} yields
 \[
  \norm{g}_\infty
  \leq \frac{\norm{g^{(r)}}_\infty}{r!} \prod_{j=1}^r \abs{x-x_j} 
  \leq \prod_{j=1}^r \abs{x-x_j} 
  < \norm{g}_\infty.
 \]
 This is a contradiction and the assertion is proven.
\end{proof}

We now construct an $\varepsilon$-detector for any $\varepsilon\in(0,1)$. 
To this end, let
\begin{equation*}
 \gamma=(1-2^{-1/d})\,\varepsilon^{1/r}.
\end{equation*}
Note that $\gamma$ is smaller than 1/2.
We choose a point set $P_0$ in $[0,1]^d$ whose dispersion is
at most $\varepsilon^{1/r}/2$ and consider
the point set $P$ in $[0,1]^d$, given by
 \begin{equation*}
 P= \set{(1-\gamma)\cdot x+\frac{\gamma j}{(r-1)d}\cdot \boldsymbol 1 
 \mid x\in P_0 \text{ and }
 j\in\IN_0 \text{ with } j\leq (r-1)d}.
\end{equation*}

\begin{lem}
\label{detector for small M}
 The point set $P$ is an \mbox{$\varepsilon$-detector} for $F_{r,M}^d$.
\end{lem}

\begin{proof}
Let $f$ be given as in \eqref{target function}.
By Lemma~\ref{almost empty interval lemma},
there are intervals $(a_i,b_i)$ in $[0,1]$ with length $\norm{f_i}_\infty^{1/r}$
containing at most $(r-1)$ zeros of $f_i$.
By the choice of $\gamma$, we have
\begin{equation*}
 \gamma 
 \leq (1-2^{-1/d}) \norm{f_i}_\infty^{1/r}.
\end{equation*}
In particular, the box
\begin{equation*}
 \widetilde B= \prod_{i\in [d]} (a_i,b_i-\gamma)
\end{equation*}
is well defined.
In fact, the volume of this box satisfies
\begin{equation*}
 \vert \widetilde B \vert
 = \prod_{i\in [d]} \brackets{\norm{f_i}_\infty^{1/r}-\gamma}
 \geq \prod_{i\in [d]} \brackets{2^{-1/d} \norm{f_i}_\infty^{1/r}}
 = \frac{\norm{f}_\infty^{1/r}}{2}
 > \frac{\varepsilon^{1/r}}{2}.
\end{equation*}
The box $\widetilde B/(1-\gamma)$ is contained in $[0,1]^d$
and even larger than $\widetilde B$.
It hence contains some $x\in P_0$.
Consequently, we have $(1-\gamma) x\in \widetilde B$
and all the points
\begin{equation*}
 x^{(j)}=(1-\gamma)\cdot x+\frac{\gamma j}{(r-1)d}\cdot \boldsymbol 1  \quad\quad\text{for }
 j\in\IN_0 \text{ with } j\leq (r-1)d
\end{equation*}
are elements of $P$.
These are $(r-1)d+1$ points
that are pairwise distinct in every coordinate
and that are all contained in the larger box 
\begin{equation*}
 B= \prod_{i\in [d]} (a_i,b_i).
\end{equation*}
Recall that each function $f_i$ has at most $(r-1)$
zeros in $(a_i,b_i)$.
By Lemma~\ref{diagonal lemma}, one of the points $x^{(j)}$
must be a non-zero of $f$.
As an example, Figure~\ref{box figure} illustrates the
case $d=2$ and $r=3$.
\end{proof}

\begin{figure}[h]
\centering
  \includegraphics[width=0.8\textwidth]{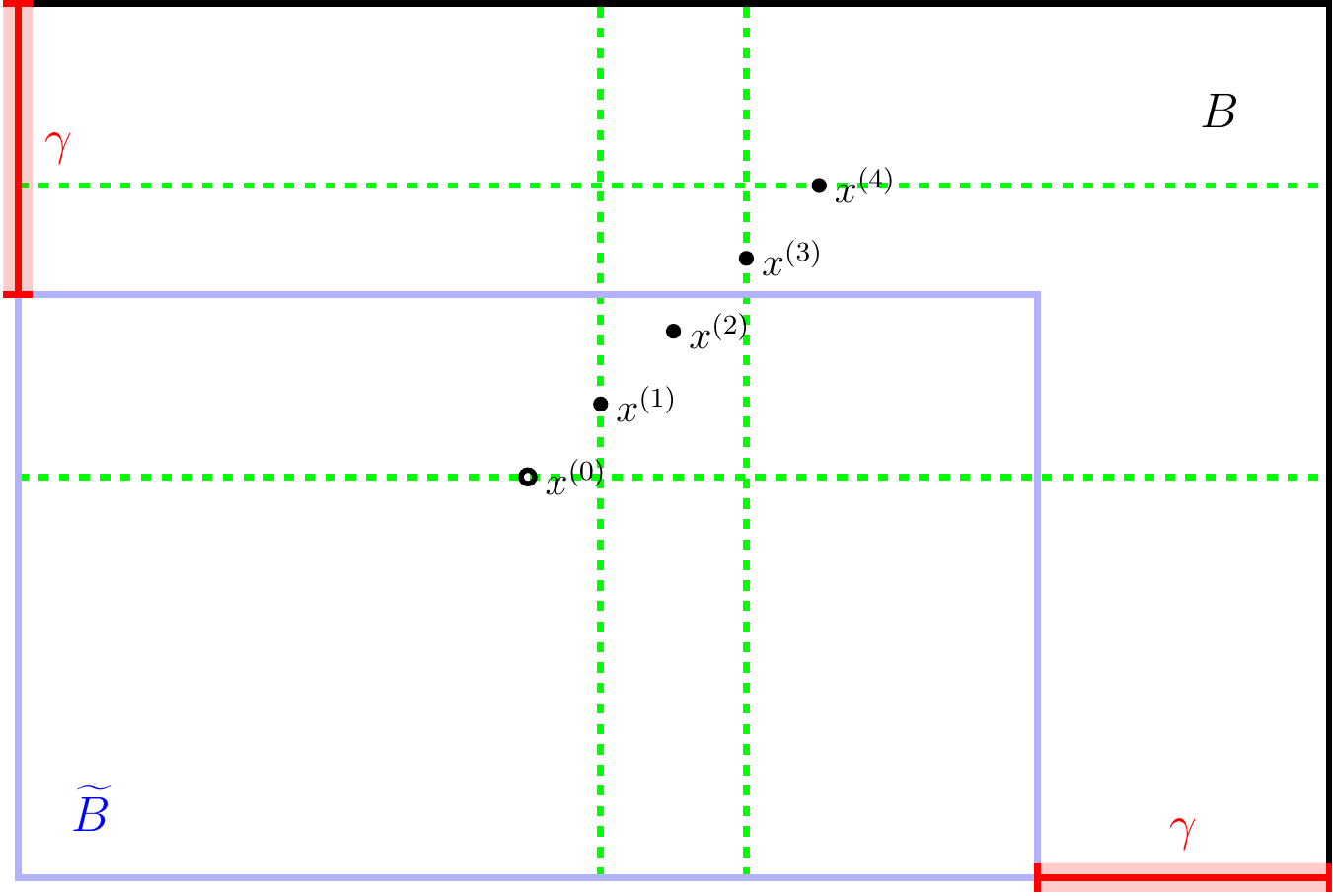}
  \label{box figure}
   \caption{The box $B$ for $(r,d)=(3,2)$.
  The dashed lines indicate the zeros of $f$ in $B$.
  Since $f$ only vanishes there,
  one of the points $x^{(j)}$ must be a non-zero.}
\end{figure}
This means that we have an $\varepsilon$-detector for $F_{r,M}^d$ with the cardinality
\begin{equation*}
 \abs{P}=((r-1)d+1) \abs{P_0},
\end{equation*}
where $P_0$ is a point set with dispersion $\varepsilon^{1/r}/2$ or less.
For example, we know from \cite{Ru16} that this can be achieved with
\begin{equation} \label{eq: pointset_Ru16}
 \abs{P_0}=\left\lceil 2^4 d\, \varepsilon^{-1/r} 
 \ln\brackets{66\varepsilon^{-1/r}} \right\rceil
\end{equation}
points.
In particular, Lemma~\ref{detector for small M} 
and Lemma~\ref{detector lemma} 
give the last statement of Theorem~\ref{thm: thm_upp_bnd}
with the constant
$c_4=85r+C_{r,M}$. 
\begin{rem}
Based on \cite{S17}, it has recently been shown 
in
\cite{UV17}
that $P_0$ can
also be chosen such that
\begin{equation*}
 \abs{P_0}=\left\lceil 2^{11} \log d\,  \log^2\brackets{2\varepsilon^{-1/r}}
 \varepsilon^{-2/r} \right\rceil.
\end{equation*}
This number is smaller than \eqref{eq: pointset_Ru16} if $d$ is large, 
but the dependence on $\varepsilon$ is worse.
\end{rem}
\begin{rem}
In contrast to the algorithms for large and moderate derivatives,
the algorithm for small derivatives is not completely explicit
since we do not know how to construct 
such point sets $P_0$.
We only know that they exist.
So far, the only construction of a point set
that achieves the desired dispersion and
only grows polynomially with the dimension
is given in~\cite{Kr17}.
However, the order of growth is proportional
to $\log(\varepsilon^{-1})$ and the
resulting algorithm would not improve on
the algorithm for moderate derivatives.
\end{rem}

\section{Lower bounds}
\label{lower bounds section}

In this section we provide lower bounds
on the complexity of the uniform approximation problem on $F_{r,M}^d$.
Together with the upper bounds from Section~\ref{algorithms section},
this proves the tractability results of Theorem~\ref{tractability results}.
First, we recall the relevant notions of tractability.
For every $n\in\IN$,
the $n$th minimal worst case error 
is given by
\[
 e(n,d):= \inf_{A_n} e(A_n),
\]
where the infimum is taken over all adaptive algorithms $A_n$ 
that use at most $n$ function values. 
The algorithms are of the form
$
 A_n(f) = \phi(f(x_1),\dots,f(x_n))
$
with $\phi \colon \mathbb{R}^n \to L_\infty$, where the $x_i\in [0,1]^d$
can be chosen adaptively, depending on the already known 
function values $f(x_1),\dots,f(x_{i-1})$, see for example \cite{No96,NW08}.
We also need the inverse of the minimal worst case error
\[
 n(\varepsilon,d):=\inf\{n\mid e(n,d)\leq \varepsilon\}.
\]
We say that the uniform recovery problem on $F_{r,M}^d$
\begin{itemize}
  \item is strongly polynomially tractable if there are $c,p>0$ such that 
 $n(\varepsilon,d)\leq c\, \varepsilon^{-p}$ 
 for all $\varepsilon\in(0,1)$ and all $d\in\IN$;
 \item is polynomially tractable if there are $c,q,p>0$ such that 
 $n(\varepsilon,d)\leq c \,\varepsilon^{-p} d^q$ 
 for all $\varepsilon\in(0,1)$ and all $d\in\IN$;
 \item is quasi-polynomially tractable if there are $c,t>0$ such that
 \[
  n(\varepsilon,d)\leq c \exp\left(t(1+\ln(\varepsilon^{-1}))(1+\ln d)\right)
 \]
for all $\varepsilon\in(0,1)$ and all $d\in\IN$;
 \item suffers from the curse of dimensionality 
  if there is some $\varepsilon>0$, $c>0$ and $\alpha>1$ 
  such that $n(\varepsilon,d)\geq c \alpha^d$.
\end{itemize}
Note that the results of Section~\ref{algorithms section}
imply that the problem is quasi-polynomially tractable if $M<2^r r!$
and polynomially tractable if $M\leq r!$.
We now provide the respective lower bounds.

\begin{proof}[The case $M\in [2^r r!,\infty)$.]
Here, the problem suffers from the curse of dimensionality.
This has already been shown in \cite[Theorem~2]{NoRu16}.
For the reader's convenience, we repeat the proof.
The function
\begin{equation*}
 g(x)=2^r (x-1/2)^r\cdot \mathbf{1}_{[0,1/2]}(x),
 \quad
 x \in [0,1]
\end{equation*}
is $r$-times differentiable with $\norm{g}_\infty=1$ and 
$\norm{g^{(r)}}_\infty \leq M$.
The same holds for the function
\begin{equation*}
 h(x)=2^r (x-1/2)^r\cdot \mathbf{1}_{[1/2,1]}(x),
 \quad
 x \in [0,1].
\end{equation*}
Hence, all functions
$f=f_{[d]}$ with $f_i \in \set{g,h}$
for $i\in[d]$
are contained in $F_{r,M}^d$ and satisfy $\norm{f}_\infty=1$.
These are $2^d$ functions with pairwise disjoint support.

Let $A$ be an algorithm and let $x_1,\dots,x_n$ 
be the sample points the algorithm uses for
the input $f_0=0$.
If $n<2^d$, there is at least one $f$ from the $2^d$ functions defined above 
that vanishes at all these points.
Therefore, the algorithm cannot distinguish $f$ and
$-f$ from $f_0$ such that
\begin{equation*}
 A(f)=A(f_0)=A(-f)
\end{equation*}
and we obtain the error bound
 \begin{equation*}
  \er(A) \geq \max\set{\norm{A(f_0)-f}_\infty,\norm{A(f_0)+f}_\infty}
  \geq \norm{f}_\infty =1.
 \end{equation*}
Hence, the problem suffers from the curse of dimensionality: 
For any $\varepsilon\in(0,1)$ we have that
$n(\varepsilon,d) \geq 2^d$.
\end{proof}

\begin{proof}[The case $M\in (r!,2^r r!)$.] 
Note that the point $x_0=\brackets{r!/M}^{1/r}$ is contained
in $(1/2,1)$. The function
\begin{equation*}
 g(x)=\frac{M(x-x_0)^r}{r!} \cdot \mathbf{1}_{[0,x_0]}(x),
 \quad
 x \in [0,1]
\end{equation*}
is $r$-times differentiable with $\norm{g}_\infty=1$
and $\norm{g^{(r)}}_\infty \leq M$.
The function
\begin{equation*}
 h(x)=\frac{M(x-x_0)^r}{r!}\cdot \mathbf{1}_{[x_0,1]}(x),
 \quad
 x \in [0,1]
\end{equation*}
is also $r$-times differentiable with $\Vert h^{(r)} \Vert_\infty \leq M$
and $\norm{h}_\infty=\abs{h(1)}$ is in $(0,1)$.
Let $k(\varepsilon,d)$ be the largest number in $[d]\cup \{0\}$  
such that $\abs{h(1)}^{k(\varepsilon,d)} > \varepsilon$.
Namely, let
\begin{equation*}
 k(\varepsilon,d)=\min\set{\kappa(\varepsilon), d}
 \quad\text{with}\quad
 \kappa(\varepsilon):=\left\lceil
 \frac{\ln(\varepsilon^{-1})}{\ln(\abs{h(1)}^{-1})}
 \right\rceil -1.
\end{equation*}
For every subset $J$ of $[d]$ with cardinality $k(\varepsilon,d)$,
the function $f=f_{[d]}$
with $f_i = g$ for $i\in J$ and
$f_i = h$ for $i\in [d]\setminus J$ is contained in $F_{r,M}^d$
and satisfies $\norm{f}_\infty>\varepsilon$.
These are $\binom{d}{k(\varepsilon,d)}$ functions 
with pairwise disjoint support.

Let $A$ be an algorithm and let $x_1,\dots,x_n$ 
be the sample points the algorithm uses for
the input $f_0=0$.
If $n<\binom{d}{k(\varepsilon,d)}$, there is at least one $f$ from the  
$\binom{d}{k(\varepsilon,d)}$ functions defined above
that vanishes at all these points.
Therefore, the algorithm cannot distinguish $f$ and
$-f$ from $f_0$, such that its error satisfies
 \begin{equation*}
  \er(A) \geq \max\set{\norm{A(f_0)-f}_\infty,\norm{A(f_0)+f}_\infty}
  \geq \norm{f}_\infty > \varepsilon.
 \end{equation*}
We obtain 
\begin{equation*}
 n(\varepsilon,d) \geq \binom{d}{k(\varepsilon,d)} 
 \geq \brackets{\frac{d}{k(\varepsilon,d)}}^{k(\varepsilon,d)}.
\end{equation*}
This implies that the problem is not polynomially tractable:
In fact, let us assume that the problem is polynomially tractable.
Then there are $c,q,p>0$ such that 
\begin{equation}
\label{trac assumption}
 n(\varepsilon,d)\leq c\, \varepsilon^{-p} d^q
\end{equation}
for all $\varepsilon\in(0,1)$ and all $d\in\IN$.
We can, however, choose $\varepsilon\in(0,1)$
such that $\kappa(\varepsilon) > q$ and hence
\begin{equation*}
 \lim\limits_{d\to \infty} \frac{n(\varepsilon,d)}{d^q}
 \geq \lim\limits_{d\to \infty}
 \frac{d^{\kappa(\varepsilon)-q}}{\kappa(\varepsilon)^{\kappa(\varepsilon)}}
 = \infty,
\end{equation*}
which contradicts the assumption \eqref{trac assumption}.
\end{proof}

\begin{proof}[The case $M\in (0,r!\rbrack$.]
First, we consider the case $r\geq 2$.
Let $A$ be an algorithm and let $x_1,\dots,x_n$ 
be the sample points the algorithm uses for
the input $f_0=0$.
Let us assume that $n\leq d$.
For each $i\in[n]$, there is a linear function $f_i$
on $[0,1]$ that vanishes at the $i$-th coordinate of $x_i$
and satisfies $\norm{f_i}_\infty = 1$.
For $i\in[d]\setminus [n]$ we set $f_i=1$.
The function $f=f_{[d]}$ is in $F_{r,M}^d$
and vanishes at all sample points.
Hence, $f$ and $-f$ cannot be distinguished from $f_0$
and the error of $A$ is at least $\norm{f}_\infty=1$.
This implies that $n(\varepsilon,d) > d$ for any $\varepsilon<1$.
In particular,  
the problem is not strongly polynomially tractable.

Now assume that $r=1$.
The previous argument does not work in this case,
since the first derivative of $f_i$ is not
necessarily bounded by $M$.
Here, we assume that 
the number of sample points of the algorithm $A$ for the input $f_0=0$
is at most $\lfloor \log_2 d \rfloor$.
By the proof of \cite[Lemma 2]{ahr}, we know that there
are two distinct coordinates $j,\ell \in [d]$ such that
the box $I_{[d]}$ does not contain any of these points,
where $I_j = [0,1/2)$, $I_\ell=(1/2,1]$
and $I_i=[0,1]$ otherwise.
The function $f=f_{[d]}$ with 
\begin{equation*}
 f_i(x)=M (x-1/2) \cdot \mathbf{1}_{I_i}(x),
 \quad x\in[0,1]
\end{equation*}
for $i\in\set{j,\ell}$ and $f_i=1$ otherwise,
is contained in $F_{r,M}^d$ and vanishes at all sample points.
Therefore, the algorithm cannot distinguish $f$ and
$-f$ from $f_0$ such that its error is at least $\norm{f}_\infty=M^2/4$.
This implies that $n(\varepsilon,d) > \lfloor \log_2 d \rfloor$ 
for any $\varepsilon<M^2/4$.
In particular, 
the problem is not strongly polynomially tractable.
\end{proof}

\begin{rem}
In \cite[Theorem~3]{NoRu16} the curse of dimensionality for 
$M\in[2^r r!,\infty)$ is also proven
for randomized algorithms. Similarly, 
one can also extend the lower bound for the case $M\in(r!,2^rr!)$ to randomized algorithms
by using a 
technique of Bakhvalov, see \cite[Section 2.2.2]{No88} for details.
\end{rem}

\noindent
{\bf Acknowledgements.} 
Both authors were supported by the Erwin Schr\"odinger International Institute for Mathematical Physics
during the program on \emph{Tractability of High Dimensional Problems and Discrepancy}.
Daniel Rudolf gratefully acknowledges support of the 
Felix-Bernstein-Institute for Mathematical Statistics in the Biosciences
(Volkswagen Foundation) and the Campus laboratory AIMS.

\raggedright{

}

\end{document}